\documentclass[12pt]{amsart}
\usepackage{amssymb, amsmath}
\newtheorem{question}{Question}[section]
\newtheorem{theorem}[question]{Theorem}
\newtheorem{lemma}[question]{Lemma}

\newtheorem{definition}[question]{Definition}

\title{Short proof of a theorem of Juh\'asz}

\author{Santi Spadaro}
\address{Department of Mathematics, Ben Gurion University of the Negev, Be'er Sheva, Israel 84105}
\email{santi@cs.bgu.ac.il, santispadaro@yahoo.com}
\thanks{Supported by the Center for Advanced Studies in Mathematics at Ben Gurion University.}

\subjclass[2000]{54A25}

\keywords{Arhangel'skii Theorem, increasing union, free sequence, elementary submodel}

\begin{document}

\begin{abstract}
We give a simple proof of the increasing strengthening of Arhangel'skii's Theorem.
\end{abstract}

\maketitle

\section{Introduction}

The pair $(X, \tau)$ denotes a Hausdorff topological space. In chapter 6 of his book \cite{J}, Istv\'an Juh\'asz proves the following theorem.

\begin{theorem} \label{arhanincr}
(\cite{J}, 6.11) Suppose $X=\bigcup_{\alpha < \lambda} X_\alpha$, where $X_\alpha \subset X_\beta$ whenever $\alpha < \beta$ and $t(X_\alpha) \cdot \psi(X_\alpha) \cdot L(X_\alpha) \leq \kappa$ for every $\alpha < \lambda$. Then $|X| \leq 2^\kappa$.
\end{theorem}

Where $\psi(X)$, $t(X)$ and $L(X)$ are respectively the \emph{pseudocharacter}, the \emph{tightness} and the \emph{Lindel\"of number} of $X$ (see \cite{E} or \cite{J}). This may be considered an \emph{increasing strengthening} (in the sense of Juh\'asz and Szentmikl\'ossy \cite{JS}) of the following theorem of Arhangel'skii and Shapirovskii.

\begin{theorem} \label{arhanshap}
(Arhangel'skii and Shapirovskii, \cite{Sh}) $|X| \leq 2^{\psi(X) \cdot t(X) \cdot L(X)}$
\end{theorem}

Juh\'asz's proof of Theorem $\ref{arhanincr}$ is five pages long. We offer a simpler and shorter proof with some help from the technique of elementary submodels. The background on this technique needed to read this paper is quite basic and can be found in the first few sections of \cite{D}.

Recall the definition of a free sequence, which already had a crucial role in the original proof \cite{A} of Arhangel'skii's famous theorem saying that the cardinality of a first-countable Lindel\"of space never exceeds the continuum (a special case of Theorem $\ref{arhanshap}$).

\begin{definition}
A set $\{x_\alpha: \alpha < \kappa\}$ is called \emph{a free sequence of length $\kappa$} if for every $\beta < \kappa$ we have $\overline{\{x_\alpha: \alpha < \beta\}} \cap \overline{\{x_\alpha : \alpha \geq \beta \}}=\emptyset$. 
\end{definition}

As usual, we let $F(X)=\sup \{|F|: F \subset X$ is a free sequence $\}$. It is known that $F(X) \leq L(X) \cdot t(X)$. Indeed, assume that $L(X) \cdot t(X) \leq \kappa$. If $X$ contained a free sequence $F$ of size $\kappa^+$ then by $L(X) \leq \kappa$, $F$ would have a complete accumulation point. But no complete accumulation point of a free sequence can lie in the closure of an initial segment of it. And this contradicts $t(X) \leq \kappa$.

\section{The main proof}

Before proving Theorem $\ref{arhanincr}$ we need a very simple lemma. Define $\Phi(X) = \sup \{L(X \setminus \{x\}): x \in X \}$.

\begin{lemma} \label{littlemma}
$\Phi(X)=L(X) \cdot \psi(X)$.
\end{lemma}

\begin{proof}
Since any open cover of $X \setminus \{x\}$ can be extended to a cover of $X$ by the addition of a single open set, we have $L(X) \leq \Phi(X)$. If for some $x \in X$ we have $L(X \setminus \{x\}) \leq \kappa$ then for every $y \neq x$ select $U_y$ such that $x \notin \overline{U_y}$. Then $\mathcal{U} =\{U_y: y \neq x \}$ covers $X \setminus \{x\}$ and hence we can find a subcover $\mathcal{V}$ having cardinality $\leq \kappa$. Then $\bigcap \{X \setminus \overline{U} : U \in \mathcal{U}\}=\{x\}$, which proves that $\psi(x,X) \leq \kappa$. So, taking suprema we have that $\psi(X) \leq \Phi(X)$, and hence $\psi(X) \cdot L(X) \leq \Phi(X)$.

To prove the other direction, fix $x \in X$ suppose that $L(X) \cdot \psi(X)=\kappa$. Let $\mathcal{U}$ be an open collection such that $|\mathcal{U}| \leq \kappa$ and $\bigcap \mathcal{U} = \{x\}$. Then $X \setminus \{x\}=\bigcup \{X \setminus U: U \in \mathcal{U}\}$ and $L(X \setminus U) \leq \kappa$ for every $U \in \mathcal{U}$. Thus $L(X \setminus \{x\}) \leq \kappa$.
\end{proof}

\begin{proof}[Proof of Theorem $\ref{arhanincr}$]
If $\lambda \leq 2^\kappa$ then we are done by Theorem $\ref{arhanshap}$, so we can assume that $\lambda=(2^\kappa)^+$. Let  $M$ be an elementary submodel of $H(\mu)$, where $\mu$ is a large enough regular cardinal, such that $[M]^\kappa \subset M$, $|M|=2^\kappa$, $\{\kappa, \tau, \lambda\} \subset M$ and $\{X_\alpha : \alpha < \lambda \} \in M$. Call a set $C \subset X$ \emph{bounded} if $|C| \leq 2^\kappa$. 

\medskip

\noindent \textbf{Claim 1.} If $C \in [X \cap M]^{\leq \kappa}$ then $\overline{C}$ is bounded.

\begin{proof}[Proof of Claim 1]
The proof of this claim is an extension of the proof of Subclaim 2 of Example 2.2 in \cite{D}. We will achieve Claim 1 if we can prove that $\overline{C} \subset X \cap M$. So, suppose that this is not true and choose $p \in \overline{C} \setminus M$. For every $\theta < \lambda$ such that $p \in X_\theta$ we can use $\psi(X_\theta) \leq \kappa$ to find open sets $\{U^\theta_\alpha: \alpha < \kappa \}$ such that $X_\theta \setminus \{p\}=\bigcup_{\alpha < \kappa} X_\theta \setminus U^\theta_\alpha$. By $L(X_\theta \setminus U^\theta_\alpha) \leq \kappa$ we can find relative open sets $\{V^\theta_{\alpha \beta}: \beta < \kappa \}$ in $X_\theta$ covering $X_\theta \setminus U^\theta_\alpha$ such that $p \notin \overline{V^\theta_{\alpha \beta}}$ for every $\alpha, \beta < \kappa$.

Then we have $\overline{C} \cap (X_\theta \setminus \{p\}) =\bigcup_{\alpha, \beta < \kappa } \overline{C^\theta_{\alpha \beta}} \cap X_\theta$ for every $\theta < \lambda$, where $C^\theta_{\alpha \beta}=V^\theta_{\alpha \beta} \cap C$.

Note now that by $\kappa$-closedness of $M$, $C^\theta_{\alpha \beta} \in M$ for every $\alpha, \beta$ and $\theta$. Moreover, since $p \notin M$ $$(\forall \theta \in \lambda \cap M)(\overline{C} \cap X_\theta \cap M=\bigcup_{\alpha, \beta < \kappa} \overline{C^\theta_{\alpha \beta}} \cap X_\theta \cap M)$$ So $$M \models (\forall \theta < \lambda) (\overline{C} \cap X_\theta=\bigcup_{\alpha, \beta < \kappa} \overline{C^\theta_{\alpha \beta}} \cap X_\theta)$$ which implies $$H(\mu) \models (\forall \theta < \lambda) (\overline{C} \cap X_\theta=\bigcup_{\alpha, \beta < \kappa} \overline{C^\theta_{\alpha \beta}} \cap X_\theta)$$ which is a contradiction because $p \in X_\theta$ for some $\theta<\lambda$.
\renewcommand{\qedsymbol}{$\triangle$}
\end{proof}

Now we claim that $X \subset M$. Suppose not and choose $p \in X \setminus M$.

\medskip

\noindent \textbf{Claim 2.} The collection $\mathcal{U}=\{U \in M \cap \tau: p \notin U\}$ is an open cover of $X \cap M$. 

\begin{proof}[Proof of Claim 2]

Fix $x \in X \cap M$ and let $\mathcal{V}=\{V \in \tau: x \notin \overline{V} \}$. Note that $\mathcal{V} \in M$ and $\mathcal{V}$ covers $X \setminus \{x\}$. Suppose you have constructed subcollections $\{\mathcal{V}_\alpha : \alpha < \beta \}$ of $\mathcal{V}$ such that $\mathcal{V}_\alpha \in M$, $|\mathcal{V}_\alpha| \leq \kappa$ for every $\alpha < \beta$ and a set $\{x_\alpha : \alpha < \beta \} \subset X \cap M$ such that $Cl_{X \setminus \{x\}}(\{x_\alpha : \alpha < \gamma \}) \subset \bigcup \bigcup_{\alpha < \gamma} \mathcal{V_\alpha}$ for every $\gamma < \beta$. By Claim 1, $Cl_{X \setminus \{x\}}(\{x_\alpha : \alpha < \beta\})$ is bounded and hence there is $\lambda_\beta < \lambda$ such that $Cl_{X \setminus \{x\}}(\{x_\alpha: \alpha < \beta\}) \subset X_{\lambda_\beta}$. By Lemma $\ref{littlemma}$, $L(Cl_{X\setminus \{x\}}(\{x_\alpha: \alpha < \beta \})) \leq \kappa$, so there is a subcollection $\mathcal{V}_\beta$ of $\mathcal{V}$ such that $|\mathcal{V}_\beta| \leq \kappa$ and $Cl_{X\setminus \{x\}}(\{x_\alpha: \alpha < \beta \}) \subset \bigcup \mathcal{V}_\beta$. If $\bigcup_{\alpha \leq \beta} \mathcal{U}_\alpha$ does not cover $X \setminus \{x\}$ then we can fix a point $x_\beta \in ((X \setminus \{x\}) \cap M) \setminus \bigcup_{\alpha \leq \beta} \mathcal{U}_\alpha$. If the induction doesn't stop at an ordinal below $\kappa^+$ then $F=\{x_\alpha : \alpha < \kappa^+\}$ is a free sequence of length $\kappa^+$ in $X \setminus \{x\}$. Since $F$ is bounded we can choose $\theta < \lambda$ such that $F \subset X_\theta$. Again by Lemma $\ref{littlemma}$ we have $L(X_\theta \setminus \{x\}) \leq \kappa$. But then $F$ cannot converge to $x$, because any $\kappa^+$ sized subset of a space of Lindel\"of number $\kappa$ has a complete accumulation point. Therefore, there is an open neighbourhood $G$ of $x$ which misses $\kappa^+$ many points of $F$ and $F \setminus G$ is a free sequence in $X$ of cardinality $\kappa^+$. But that contradicts $F(X_\theta) \leq \kappa$.

So there is a subcollection $\mathcal{W} \subset \mathcal{V}$ such that $|\mathcal{W}| \leq \kappa$ and $X \setminus \{x\} \subset \bigcup \mathcal{W}$. Now, by elementarity we can take $\mathcal{W} \in M$, but then we also have $\mathcal{W} \subset M$. Let now $W \in \mathcal{W}$ such that $p \in W$. Then the set $U=X \setminus \overline{W} \in M$ is an open neighbourhood of $x$ such that $p \notin U$.
\renewcommand{\qedsymbol}{$\triangle$}
\end{proof}

Suppose that for some $\beta < \kappa^+$ we have constructed a set $\{x_\alpha : \alpha < \beta \} \subset X \cap M$ and subcollections $\{\mathcal{U}_\alpha : \alpha < \beta\}$ of $\mathcal{U}$ such that $\mathcal{U}_\alpha \in M$, $|\mathcal{U}_\alpha| \leq \kappa$ and $\overline{\{x_\gamma : \gamma < \alpha\}} \subset \bigcup \bigcup_{\gamma< \alpha} \mathcal{U}_\gamma$, for every $\alpha < \beta$. Since $\overline{\{x_\alpha: \alpha < \beta \}}$ is bounded we have $L(\overline{\{x_\alpha : \alpha < \beta \}}) \leq \kappa$ and hence we can find a subcollection $\mathcal{U}_\beta$ of $\mathcal{U}$ of cardinality not exceeding $\kappa$ such that $\overline{\{x_\alpha: \alpha < \beta\}} \subset \bigcup \mathcal{U}_\beta$. If $\bigcup_{\alpha \leq \beta} \mathcal{U}_\alpha$ does not cover $X$ then we can pick a point $x_\beta \in X \cap M \setminus \bigcup \bigcup_{\alpha \leq \beta} \mathcal{U}_\alpha$. If we didn't stop then $\{x_\alpha : \alpha < \kappa^+ \}$ would be a free sequence of size $\kappa^+$ in $X$. But that can't happen since $\{x_\alpha : \alpha < \kappa^+\}$ is bounded. So there is a subcollection $\mathcal{V} \subset \mathcal{U}$ such that $|\mathcal{V}| \leq \kappa$ and $X \cap M \subset \bigcup \mathcal{V}$. But since $\mathcal{V} \in M$ we have that $M \models X \subset \bigcup \mathcal{V}$ and hence also $H(\mu)$ thinks that there is $V \in \mathcal{V}$ such that $p \in V$. But that's a contradiction.
\end{proof}

\section{Remarks and acknowledgements}
Note that in the proof of Theorem $\ref{arhanincr}$ we never used $t(X) \leq \kappa$, but only $F(X) \leq \kappa$. Therefore we actually proved the increasing strengthening of the following theorem. 

\begin{theorem} \label{arhanfree}
$|X| \leq 2^{\psi(X) \cdot F(X) \cdot L(X)}$.
\end{theorem}

Theorem $\ref{arhanfree}$ was independently proved by Juh\'asz \cite{JP} and the author in his dissertation \cite{S}.

Besides in the entire chapter 6 of Juh\'asz's book and in Juh\'asz and Szentmikl\'ossy's paper \cite{JS}, cardinal functions on unions of chains were also extensively studied by Mikhail Tkachenko, both in a general topological context (\cite{Tk1}, \cite{Tk2}, \cite{Tk3}) and with Torres Falc\'on in the presence of a topological group structure (\cite{Tk4}, see also Torres Falc\'on's paper \cite{TF}). The study of unions of chains is very useful in obtaining reflection theorems for cardinal functions, as proved by Hajnal and Juh\'asz's paper \cite{HJ} and Hodel and Vaughan's paper \cite{HV}.

In Chapter 6 of his dissertation \cite{S} the author claimed to have a simple proof of the increasing strengthening of Theorem $\ref{arhanfree}$. That proof however still relied on two lemmas from Juh\'asz's book \cite{J}, while the present one is more self-contained and even shorter. The author would like to thank his PhD advisor, Gary Gruenhage, for valuable discussion, Istv\'an Juh\'asz for sending him the seminar slides where Theorem $\ref{arhanfree}$ was proved, Mikhail Tkachenko for bringing the reference \cite{Tk3} to his attention and the Center for Advanced Studies in Mathematics at Ben Gurion University for financial support.

\end{document}